\documentclass[oneside,a4paper,12pt]{article}
\usepackage{mathtools}
\usepackage{amsmath}
\allowdisplaybreaks[4]
\usepackage{amsthm}
\usepackage{amssymb}
\usepackage{mathrsfs}
\usepackage[all]{xy}
\usepackage{tikz,tikz-cd}
\usepackage[bottom]{footmisc}
\usepackage{bm}
\usepackage{fancyhdr}%
\usepackage{geometry}
\usepackage{hyperref}
\usepackage{cite}
\usepackage{enumitem}
\usepackage{verbatim}
\setlist[enumerate]{label=(\arabic*)}

\newtheorem{thm}{Theorem}[section]
\newtheorem{cor}[thm]{Corollary}
\newtheorem{lem}[thm]{Lemma}
\newtheorem{prop}[thm]{Proposition}
\theoremstyle{definition}
\newtheorem{defn}[thm]{Definition}
\theoremstyle{remark}
\newtheorem{rem}[thm]{Remark}

\numberwithin{equation}{section}

\title{\textbf{RCD(0,\textit{N})-spaces with small linear diameter growth}}
\author{Xin Qian\footnote{School of Mathematical Sciences, Fudan University, Shanghai
		200433, People's Republic of China. e-mail: xqian22@m.fudan.edu.cn}}
 
\date{}
\usepackage{fancyhdr}
\providecommand{\keywords}[1]
{
	\small	
	\textbf{Keywords:} #1
}
\begin{document}
\maketitle
\begin{abstract}
In this paper, we study some structure properties on the fundamental group of RCD($0,N$) spaces.\ Our main result generalizes earlier work of Sormani \cite{sormani2000nonnegative} on Riemannian manifolds with nonnegative Ricci curvature and small linear diameter growth.\ We prove that the fundamental group is finitely generated if assuming small linear diameter growth on RCD($0,N$) spaces.
		 
			
\keywords{RCD($0,N$), finitely generated, fundamental group}
\end{abstract}

\section{Introduction} 
In recent years, the theory of RCD($K,N$) spaces has a remarkable development. After Lott-Villani \cite{lott2009ricci} and Sturm \cite{sturm2006geometry, sturm2006geometry2} introduced the curvature dimension condition CD($K,N$) independently, the notion of RCD($K,N$) spaces was proposed and analyzed in \cite{gigli2015differential, ambrosio2019nonlinear}, as a finite dimensional counterpart of RCD($K,\infty$) which was first introduced in \cite{ambrosio2014metric}.\ Roughly speaking, a CD($K,N$) space is a metric measure spaces $(X,d,m)$ with Ricci curvature bounded from below by $K\in\mathbb{R}$ and dimension bounded from above by $N\in[1,\infty]$, and RCD($K,N$) spaces are those infinitesimally Hilbertian CD($K,N$) spaces.\ One of the disadvantages of CD($K,N$) for finite $N$ is the lack of a local-to-global property.\ To this aim, Bacher and Sturm \cite{bacher2010localization} introduced the reduced  curvature dimension condition CD$^*$($K,N$) and as before, RCD$^*$($K,N$) spaces are infinitesimally Hilbertian CD$^*$($K,N$) spaces.\\
\indent The motivation of studying RCD($K,N$) (resp.\ RCD$^*$($K,N$)) spaces is to single out the “Riemannian” class in CD($K,N$) (resp.\ CD$^*$($K,N$)) spaces, which excludes Finsler manifolds.\ Therefore, it is natural to expect that some analytical and topological properties of Riemannian manifolds also hold on RCD($K,N$) or RCD$^*$($K,N$) spaces.\\
\indent It is well-known that for Riemannian manifolds, the Ricci curvature controls the fundamental group very well.\ In this area, one famous open problem in the past was that an open manifold with nonnegative Ricci curvature has a finitely generated fundamental group, conjectured by Milnor \cite{milnor1968note} in 1968.\ This conjecture was of great interests until Bru\`e-Naber-Semola \cite{brue2023fundamental} constructed a counterexample $M^{7}$ with $Ric\geqslant0$ such that $\pi_{1}(M^{7})=\mathbb{Q}/\mathbb{Z}$.\ However, it is still worth studying under what conditions Milnor Conjecture holds.\ Some results in this direction have been accomplished by Anderson \cite{anderson1990topology}, Li \cite{li1986large} and Sormani \cite{sormani2000nonnegative} among others.\ Anderson and Li proved that if a manifold with $Ric\geqslant0$ has Euclidean volume growth, then the fundamental group is finite and this result has been extended to an RCD($0,N$) space $(X,d,m)$ by Mondino-Wei \cite{mondino2019universal}.\ A key technical point to get information on the fundamental group $\pi_{1}(X)$ is to study the universal cover and the group of deck transformations, denoted by $G(X)$, which is a quotient of the fundamental group $\pi_{1}(X)$ and called revised fundamental group in \cite{sormani2001hausdorff}.\\
\indent Moreover, in a more recent paper \cite{wang2023rcd}, Wang proved that any RCD($K,N$) space $(X,d,m)$ is semi-locally simply connected, which implies that the universal cover $\tilde{X}$ is simply connected and $G(X)$ is isomorphic to $\pi_{1}(X)$.\ Hence, the structure properties derived by Mondino-Wei \cite{mondino2019universal} on $G(X)$ hold on the fundamental group $\pi_{1}(X)$.\ We will review this point in the next section.

In this paper we extend the result of Sormani \cite{sormani2000nonnegative} to RCD($0,N$) spaces and the main theorem is stated as following. 
\begin{thm}\label{thm1.1}
   Let $(X,d,m)$ be an RCD(0,N) space for some $N\in[1,\infty)$.\ Then,
   \begin{enumerate}
    \item If $N>2$, and X has small linear diameter growth, i.e.,
    \[
    \limsup_{r \to \infty}
     \frac{diam(\partial B_{r}(p))}{r}
    < 4S_N,
    \]	
    where \[ S_N =
    \dfrac{N}{4(N-1)} \dfrac{1}{3^N} \left(\dfrac{N-2}{N-1}\right)^{N-1},\]
    then the fundamental group $\pi_{1}(X)$ is finitely generated.
    \item If $N=2$, then the fundamental group $\pi_{1}(X)$ is finitely generated.
    \item If $1\leqslant N<2$, then the fundamental group $\pi_{1}(X)$ is trivial.
    \end{enumerate}
\end{thm}
\indent To prove this result, we extend the Halfway Lemma and Uniform Cut Lemma established by Sormani in \cite{sormani2000nonnegative} to a non-smooth setting and the spirit of the proof is similar.\ Let us point out that Kitabeppu and Lakzian proved a similar result in \cite{kitabeppu2015non}, under additional non-branching and semi-locally simple connectedness assumptions.\ In our argument, we drop the non-branching assumption by carefully applying the excess estimate established for RCD($K,N$) spaces in \cite{gigli2014abresch}.\ It can be seen in the proof that the excess estimate can prevent the structure of an RCD($0,N$) space being too wild.\ Also, the semi-locally simple connectedness is necessary and sufficient for the existence of a simply connected universal cover, but by Wang's result \cite{wang2023rcd}, we can directly drop this assumption.\\
\indent Furthermore, we prove the finite generation of $\pi_{1}(X)$ without assuming small linear diameter growth in the case $N=2$, which differs from Kitabeppu and Lakzian's result in \cite{kitabeppu2015non}.\ Our results are more similar to the situation on manifolds.\ Notice that for 2-dimensional manifolds, $Ric\geqslant0$ is equivalent to $sec\geqslant0$, which implies that the the fundamental group is finitely generated (see \cite{gromov1978manifolds}).\\
\indent Also, since the Ricci-limit space is RCD space, it is straightforward to obtain the following corallary.
    \begin{cor}\label{cor1.2}
      Let $(M_{i}^{n},g_{i},p_{i})$ be a sequence of n-dimensional Riemannian manifold with $Ric_{g_{i}}\geqslant-\delta_{i}\to0$ and $(M_{i}^{n},g_{i},p_{i})\xrightarrow{pGH}(X,d,p)$.\ Then,
      \begin{enumerate}
      	\item If $n\geqslant3$, and X has small linear diameter growth, i.e.,
      	\[
      	\limsup_{r \to \infty}
      	\frac{diam(\partial B_{r}(p))}{r}
      	< 4S_n,
      	\]	
      	where $S_n=\dfrac{n}{4(n-1)} \dfrac{1}{3^n} \left(\dfrac{n-2}{n-1}\right)^{n-1}$, then the fundamental group $\pi_{1}(X)$ is finitely generated.
      	\item If $n=2$, then the fundamental group $\pi_{1}(X)$ is finitely generated.
      \end{enumerate}
    \end{cor}
\indent Moreover, if an RCD($0,N$) space has linear volume growth, then by Huang \cite{huang2020almost}, its diameter growth is sublinear, which was proved on manifolds by Sormani \cite{sormani2000almost}. Thus, we get the following corollary.
\begin{cor}\label{cor1.3}
	Let $(X,d,m)$ be an RCD(0,N) space with $m(B_{r}(p))\leqslant Cr$ for some positive constant $C$ and $N\in[1,\infty)$.\ Then the fundamental group $\pi_{1}(X)$ is finitely generated.
\end{cor}
\indent The paper is organized as follows. In section \ref{2}, we recall the definition of lower Ricci curvature bounds on metric measure spaces and review some basic properties and useful results. In particular, we will discuss Mondino and Wei's work on the universal cover of RCD($K,N$) spaces \cite{mondino2019universal}. In section \ref{3}, we generalize Sormani's technical lemmas to RCD($0,N$) spaces, and the proof of Theorem \ref{thm1.1}, Corollary \ref{cor1.2} and \ref{cor1.3} are presented in section \ref{4}.

\section{Preliminaries}\label{2}
Throughout this paper, $(X,d,m)$ is a metric measure space (mms. for short) where $(X,d)$ is a complete and separable geodesic metric space and $m$ is a locally finite nonnegative Borel measure with supp$m=X$.\ We also assume $X$ is not a point.
\subsection{RCD$^*$($K,N$) spaces}	
	In this subsection, we recall some basic definitions and properties of metric measure spaces with lower Ricci curvature bounds.\\
\indent We denote by $\mathcal{P}(X)$ the set of Borel probability measures on $(X,d)$ and by $\mathcal{P}_{2}(X)\subset\mathcal{P}(X)$ the subset of probability measures with finite second moment, i.e.
	\[\mathcal{P}_{2}(X)=\left\lbrace \mu\in\mathcal{P}(X):\int_{X}d^{2}(\cdot,x_{0})d\mu<\infty,\  \text{for some}\ x_{0}\in X\right\rbrace.\]
For $\mu_{0},\mu_{1}\in\mathcal{P}_{2}(X)$, the $W_{2}$-distance $W_{2}(\mu_{0},\mu_{1})$ is defined by
	\begin{align}
		W_{2}^{2}(\mu_{0},\mu_{1})=\inf\limits_{\alpha}\int_{X\times X}d^{2}(x,y)d\alpha(x,y) \label{eq2.1}
	\end{align}
where the infimum is taken over all $\alpha\in\mathcal{P}(X\times X)$ such that $\pi_{\#}^1\alpha=\mu_{0},\pi_{\#}^2\alpha=\mu_{1}$.\\
\indent Note that $(\mathcal{P}_{2}(X),W_{2})$ is a geodesic space provided that $(X,d)$ is a geodesic space.\ Then we define the evaluation map $e_{t}:C([0,1];X)\to X$ as
          $$e_{t}(\gamma):=\gamma_{t},\qquad \forall\gamma\in C([0,1];X).$$
We will denote by Geo($X$) ($\subset\!C([0,1];X)$) the space of (constant speed minimizing) geodesics on $(X,d)$ endowed with the sup distance.\\ 
\indent Given $\mu_{0},\mu_{1}\in\mathcal{P}_{2}(X)$, let OptGeo($\mu_{0},\mu_{1}$) be the space of all $\pi\in\mathcal{P}$(Geo($X$)) for which $(e_{0},e_{1})_{\#}\pi$ is a minimizer in \eqref{eq2.1} and recall that $(t\to\mu_{t})_{t\in[0,1]}$ is in Geo($\mathcal{P}_{2}(X)$) if and only if there exists $\pi\in$OptGeo($\mu_{0},\mu_{1}$) such that $(e_{t})_{\#}\pi=\mu_{t}$ for all $t\in[0,1]$.\\
\indent Now, let us introduce the so-called reduced curvature dimension condition CD$^*$($K,N$), coming from \cite{bacher2010localization}. For $K,N\in\mathbb{R}$ with $N\geqslant0$, and $(t,\theta)\in[0,1]\times\mathbb{R}_{+}$, we set
\begin{align}
\sigma_{K,N}^{(t)}(\theta):=
\begin{cases}
	\infty & \text{if } K\theta^{2}\geqslant N\pi^{2}\ \text{and}\ K\theta^{2}>0,\\
	\dfrac{\sin(t\theta\sqrt{K/N})}{\sin(\theta\sqrt{K/N})} & \text{if } 0<K\theta^{2}<N\pi^{2},\\
	t  & \text{if } K\theta^{2}<0\ \text{and}\ N=0,\ \text{or if}\ K\theta^{2}=0, \\
	\dfrac{\sinh(t\theta\sqrt{-K/N})}{\sinh(\theta\sqrt{-K/N})} & \text{if } K\theta^{2}<0\ \text{and}\ N>0,
\end{cases}
\end{align}
and let
\begin{align}
	\tau_{K,N}^{(t)}(\theta):=t^{1/N}\sigma_{K,N-1}^{(t)}(\theta)^{(N-1)/N},
\end{align}
for all $K\in\mathbb{R}, N\in[1,\infty), (t,\theta)\in[0,1]\times\mathbb{R}_{+}$.
\begin{defn}[Reduced curvature dimension condition]\label{defn2.1}
	Let $K\in\mathbb{R}$ and $N\in[1,\infty)$.\ We say that a mms.\ $(X,d,m)$ is a CD$^*$($K,N$) space if for any two measures $\mu_{0},\mu_{1}\in\mathcal{P}(X)$ with bounded support,\ there exists a measure $\pi\in$ OptGeo($\mu_{0},\mu_{1}$) such that for any $t\in[0,1]$ and $N'\geqslant N$, we have
	\begin{align}
		\int\rho_{t}^{1-\frac{1}{N'}}dm\geq\int\sigma_{K,N'}^{(1-t)}(d(\gamma_{0},\gamma_{1}))\rho_{0}^{-\frac{1}{N'}}(\gamma_{0})+\sigma_{K,N'}^{(t)}(d(\gamma_{0},\gamma_{1}))\rho_{1}^{-\frac{1}{N'}}(\gamma_{1}) d\pi(\gamma)\label{eq2.4}
	\end{align}
where we have written $(e_{t})_{\#}\pi=\rho_{t}m+\mu_{t}^{s}$ with $\mu_{t}^{s}\perp m$, for all $t\in[0,1]$.
\end{defn}
\begin{rem}\label{rem2.2}
	On a CD$^*$($K,N$) space $(X,d,m)$, a natural version of Bishop-Gromov volume comparison holds (see \cite{bacher2010localization} for precise statement) and this leads to the properness of $(X,d,m)$ (i.e.\ closed bounded sets in $X$ are compact).
\end{rem}
\begin{rem}\label{rem2.3}
	The original curvature dimension condition, denoted by CD($K,N$), has the same definition as CD$^*$($K,N$) except that the coefficients $\sigma_{K,N'}^{(1-t)}(d(\gamma_{0},\gamma_{1}))$ and $\sigma_{K,N'}^{(t)}(d(\gamma_{0},\gamma_{1}))$ in \eqref{eq2.4} are replaced by $\tau_{K,N'}^{(1-t)}(d(\gamma_{0},\gamma_{1}))$ and $\tau_{K,N'}^{(t)}(d(\gamma_{0},\gamma_{1}))$ respectively. In general, CD$^*$($K,N$) is weaker than CD($K,N$), while in the case $K=0$ which we will mainly consider in this paper, these two notions are identical.
\end{rem}
\indent It is possible to see that Finsler manifolds are allowed as CD$^*$($K,N$) spaces.\ In order to single out the “Riemannian class”, the CD$^*$($K,N$) condition might be strengthened by requiring additionally that the Sobolev space $W^{1,2}(X,d,m)$ is Hilbert, inspired by the fact that a smooth Finsler manifold is Riemannian if and only if the space $W^{1,2}$ is Hilbert.\ We briefly review the definition of Sobolev space for mms.\ $(X,d,m)$ below.\\
\indent Recall that the Cheeger energy $Ch:L^{2}(X,m)\to[0,\infty]$ is defined through
\begin{align}
	Ch(f):=\inf\left\lbrace \liminf\limits_{n\to\infty}\int_{X}(\text{lip}f_{n})^{2}dm: f_{n}\in\text{Lip}_{b}(X)\cap L^{2}(X,m),\ f_{n}\xrightarrow{L^{2}}f\right\rbrace \label{eq2.5}
\end{align}
where, $\text{lip}f(x):=\limsup\limits_{y\to x}\frac{|f(x)-f(y)|}{d(x,y)}$.
Then, we define
\begin{align}
	W^{1,2}(X)=W^{1,2}(X,d,m):=\left\lbrace f\in L^{2}(X):\ Ch(f)<\infty\right\rbrace.
\end{align}
By looking at the optimal approximating sequence in \eqref{eq2.5}, one can identify a canonical object $|\nabla f|$, called minimal relaxed gradient, which provides the integral presentation
\[Ch(f)=\int_{X}|\nabla f|^{2}dm,\qquad \forall f\in W^{1,2}(X).\] See \cite{ambrosio2014calculus} for more details on this topic.\\
\indent The Sobolev space $W^{1,2}(X)$ endowed with the norm $||f||_{W^{1,2}}^{2}=||f||_{L^{2}}^{2}+Ch(f)$ is a Banach space, but it is not Hilbert in general.\ If $W^{1,2}(X,d,m)$ is Hilbert, then we say that $(X,d,m)$ is infinitesimally Hilbertian.
\begin{defn}\label{defn2.4}
An RCD$^*$($K,N$) (resp.\ RCD($K,N$)) space $(X,d,m)$ is an infinitesimally Hilbertian CD$^*$($K,N$) (resp.\ CD($K,N$)) space.
\end{defn}
\noindent A basic property of RCD$^*$($K,N$) spaces is the stability under pointed measured Gromov-Hausdorff (pmGH for short) convergence.\ See \cite{gigli2015convergence, erbar2015equivalence} for a proof.
\begin{prop}[Stability]\label{prop2.5}
	Let $K\in\mathbb{R}\ and\ N\in[1,\infty)$.\ If $((X_{n},d_{n},m_{n}))_{n\in\mathbb{N}}$ is a sequence of RCD$^*$($K,N$) spaces with $(X_{n},d_{n},m_{n})\xrightarrow{pmGH}(X,d,m)$, then $(X,d,m)$ is also an RCD$^*$($K,N$) space.
\end{prop}
Finally, we state a few properties of RCD($0,N$) spaces (the first one is proved in \cite{gigli2014abresch}, the second in \cite{gigli2014overview}, the third in \cite{huang2020almost} and the fourth in \cite{huang2016noncompact}).\ Note that RCD($0,N$) condition is exactly the same as RCD$^*$($0,N$) condition.
\begin{thm}[Abresch-Gromoll excess estimate]\label{thm2.6}
	Let $(X,d,m)$ be an RCD(0,N) space with $N\in(1,\infty)$.\ Fix $p,q\in X$ and a minimizing geodesic $\gamma$ joining them.\ Define 
	\[l(x):=\min\{d(x,p),d(x,q)\}\ and\ h(x):=\min\limits_{t}d(x,\gamma_{t}).\]
	Then for any $x\in X$ with $l(x)>h(x)$, we have
\begin{align}
	e(x)\leqslant
	\begin{cases}
		2\frac{N-1}{N-2}\left(\frac{N-1}{N}\frac{h^{N}(x)}{l(x)-h(x)}\right)^{\frac{1}{N-1}} & \text{if } N>2,\\ \\
		\frac{N-1}{2-N}\frac{h^{2}(x)}{l(x)-h(x)}  & \text{if } 1<N<2, \\ \\
		a(x)h(x)\left(\frac{1}{1+\sqrt{1+a(x)^{2}}}+\log\frac{1+\sqrt{1+a(x)^{2}}}{a(x)}\right)  & \text{if } N=2.
	\end{cases}
\end{align}
where $a(x):=\frac{h(x)}{2(l(x)-h(x))}$ and $e(x):=d(x,p)+d(x,q)-d(p,q)$ is the excess function w.r.t p and q.
\end{thm}
Note that in \cite{gigli2014abresch}, Gigli and Mosconi proved excess estimates for all RCD($K,N$) spaces with $K\leqslant0$, but for our purposes, we only need the case $K=0$ here.
\begin{thm}[Splitting]\label{thm2.7}
	Let $(X,d,m)$ be an RCD(0,N) space with $N\in[1,\infty)$.\ If $(X,d)$ contains a line, then
	$(X,d,m)$ is isomorphic to $(X'\times\mathbb{R},d'\times d_{E},m'\times\mathcal{L}^{1})$, where $d_{E}$ is the Euclidean metric, $\mathcal{L}^{1}$ is the Lebesgue measure and $(X',d',m')$ is an RCD(0,N-1) space if $N\geqslant2$ and a singleton if $N<2$.
\end{thm}
\begin{thm}\label{thm2.8}
	Let $(X,d,m)$ be an RCD(0,N) space with $m(B_{r}(p))\leqslant Cr$ for some positive constant $C$ and $N\in[1,\infty)$.\ If $(X,d,m)$ does not split, then
	\[
	\lim_{r \to \infty}
	\frac{diam(\partial B_{r}(p))}{r}=0.
	\]	
\end{thm}
\begin{thm}\label{thm2.9}
	Let $(X,d,m)$ be a noncompact RCD(0,N) space with $N\in[1,\infty)$, then for every $p\in X$, there exists a constant $C=C(N,m(B_{1}(p)))$ such that
	\[
	m(B_{r}(p))\geqslant Cr.
	\]	
\end{thm}
\subsection{The topology on metric measure spaces}	
We first recall the definition of the universal cover of a metric space \cite{spanier1989algebraic}.
\begin{defn}[Universal cover of a metric space]\label{defn2.10}
	Let $(X,d)$ be a connected metric space.\ We say that a connected metric space $(\tilde{X},\tilde{d})$ is a universal cover of $(X,d)$ if $(\tilde{X},\tilde{d})$ is a cover for $(X,d)$ with the following property: for any cover $(\bar{X},\bar{d})$ of $(X,d)$, there is a commutative diagram formed by a continuous map $f:(\tilde{X},\tilde{d})\to(\bar{X},\bar{d})$ and the two covering projections onto $X$:
	\begin{displaymath}
		\xymatrix{
		\tilde{X} \ar[r]^{f}
	              \ar[dr]_{p_{1}}
	    &  \bar{X} \ar[d]^{p_{2}} \\
	    &  {X}
             . }
	\end{displaymath}
\end{defn}
\begin{rem}\label{rem2.11}
	The universal cover may not exist in general.\ However, if it exists then it is unique.\ Moreover, if a space is locally path connected and semi-locally simply connected (i.e., for all $x\in X$, there is a neighborhood $U_{x}$ such that any loop in $U_{x}$ is contractible in $X$), then it has a simply connected universal cover.\ On the other hand, the universal cover of a locally path connected space may not be simply connected (See \cite{spanier1989algebraic}). 
\end{rem}
If a space has a universal cover, then one can consider the revised fundamental group introduced in \cite{sormani2001hausdorff}.\ We first recall that for a covering $\pi:Y\to X$, a deck transformation is a homeomorphism $h:Y\to Y$ such that $\pi\circ h=\pi$ and all deck transformations form a group $Deck(Y,X)$.
\begin{defn}[Revised fundamental group]\label{defn2.12}
	 $(X,d)$ is a metric space which admits a universal cover $(\tilde{X},\tilde{d})$.\ Then the revised fundamental group of $X$, denoted by $G(X)$, is the group of deck transformations $Deck(\tilde{X},X)$.
\end{defn}
Notice that $G(X)$ can be seen as a quotient of $\pi_{1}(X,x)$ and the trivial element in $G(X)$ is represented by those loops in $X$ based at $x$, which are still loops when lifted to $\tilde{X}$.\ Thus if the universal cover $\tilde{X}$ is simply connected, then $G(X)$ is isomorphic to $\pi_{1}(X)$.\ Before Wang \cite{wang2023rcd} proved the semi-locally simply connectedness for $\text{RCD}^{*}(K,N)$ spaces, the existence of the universal cover of an RCD$^*$($K,N$) space was confirmed by Mondino and Wei \cite{mondino2019universal}.

\begin{thm}\label{thm2.13}
	Let $(X,d,m)$ be an RCD$^*(K,N)$ space with $K\in\mathbb{R}\ and\ N\in(1,\infty)$.\ Then
	$(X,d,m)$ admits a universal cover $(\tilde{X},\tilde{d},\tilde{m})$ which is also an RCD$^*(K,N)$ space.
\end{thm}
\begin{rem}\label{rem2.14}
	By Mondino and Wei's contruction, $\tilde{X}$ naturally inherits the length structure of the base $X$ and $\tilde{X}$ is locally isometric to $X$.\ Then by completeness and locally compactness, $X$ is a geodesic space.\ In our argument below, we will always assume the universal cover $\tilde{X}$ to be geodesic.
\end{rem}
Mondino and Wei \cite{mondino2019universal} also derived several structure properties on the revised fundamental group of an RCD$^*$($K,N$) space in the similar manner as Riemannian geometry.\ We present one of them, which is an extension of the celebrated result by Cheeger-Gromoll \cite{cheeger1971splitting} for compact manifolds with nonnegative Ricci curvature.
\begin{thm}\label{thm2.15}
	Let $(X,d,m)$ be a compact RCD(0,N) space with $N\in(1,\infty)$.\ Then the revised fundamental group $G(X)$ contains a finite normal subgroup $\phi\lhd G(X)$ such that $G(X)/\phi$ contains a  subgroup $\mathbb{Z}^{k}$ of finite index.
\end{thm}
\begin{rem}\label{rem2.16}
	By Theorem \ref{thm2.15}, the revised fundamental group of a compact RCD($0,N$) space is finitely generated (see also Proposition 2.25 in \cite{mondello2022upper}).\ Thus, for our purposes, we only need to consider the noncompact case when studying the finite generation on the (revised) fundamental group of an RCD($0,N$) space.
\end{rem}
Finally, we point out that it has been proved by Wang in \cite{wang2023rcd} that any $\text{RCD}^{*}$ space $(X,d,m)$ is semi-locally simply connected, which generalizes the same author's result in \cite{wang2021ricci}.
\begin{thm}\label{thm2.17}
Let $(X,d,m)$ be an $\text{RCD}^{*}(K,N)$ space with $K\in\mathbb{R}\ and\ N\in(1,\infty)$.\ Then $X$ is semi-locally simply connected.\ In particular, the universal cover $\tilde{X}$ is simply connected and $G(X)$ is isomorphic to $\pi_{1}(X)$.
\end{thm}
By Theorem \ref{thm2.17}, in the RCD setting, we get rid of the notion of revised fundamental group and the revised fundamental group $G(X)$ in Theorem \ref{thm2.15} can be replaced by $\pi_{1}(X)$.\ In the remaining part of this paper, we still use the notation $G(X)$ to denote the deck transformation group, which acts on $\tilde{X}$ discretely.\ Note that in this paper, we typically consider RCD spaces and $G(X)$ is isomorphic to $\pi_{1}(X)$ in most of our context.

\subsection{Structure of RCD($K,N$) spaces}
The main purpose of this subsection is to provide some metric measure structure theory of RCD spaces, which we will need in the proof of Theorem \ref{thm1.1} for the case $N=2$.\ We use the notion RCD($K,N$) instead of RCD$^{*}$($K,N$) in this subsection, since most previous works reviewed in this subsection selected this stronger notion, though it seems that all these results hold on RCD$^{*}$($K,N$) spaces.\\
\indent Given an RCD($K,N$) space $(X,d,m)$ with $x\in X$, we first recall the notion of tangent cones.
\begin{defn}[tangent cones]\label{defn2.18}
	We say that a pointed metric measure space $(Y,d_{Y},m_{Y},y)$ is a tangent cone of $(X,d,m)$ at $x$ if there exists a sequence $r_{i}\to0^{+}$ such that 
	\[(X,r_{i}^{-1}d,m(B_{r_{i}}(x))^{-1}m,x)\xrightarrow{pmGH}(Y,d_{Y},m_{Y},y).\]
	The collection of all tangent cones of $(X,d,m)$ at $x$ is denoted by $\text{Tan}(X,d,m,x)$.
\end{defn}
\indent A compactness result on RCD($K,N$) spaces yields that $\text{Tan}(X,d,m,x)$ is non-empty for any $x\in X$ (see Chapter 27 in \cite{villani2009optimal} for instance).\ We are now in the position to introduce the notions of $k$-regular set and essential dimension as follows.
\begin{defn}[$k$-regular set]\label{defn2.19}
    For any integer $k\in[1,N]$, we denote by $\mathcal{R}_{k}$ the set of all points $x\in X$ such that $\text{Tan}(X,d,m,x)=\left\lbrace(\mathbb{R}^{k},d_{eucl},(\omega_{k})^{-1}\mathcal{H}^{k},0^{k})\right\rbrace$, where $\omega_{k}$ is the volume of the unit ball in $\mathbb{R}^{k}$.\ We call $\mathcal{R}_{k}$ the $k$-regular set of $X$.
\end{defn}
\indent The following result is proved by Bru\`e-Semola in \cite{brue2020constancy}.
\begin{thm}\label{thm2.20}
	Let $(X,d,m)$ be an RCD$(K,N)$ space with $K\in\mathbb{R}\ and\ N\in(1,\infty)$.\ Then there exists a unique integer $k\in[1,N]$, called the essential dimension of $(X,d,m)$, denoted by $\dim_{ess}(X)$, such that $m(X\setminus\mathcal{R}_{k})=0$.
\end{thm}
When the essential dimension reaches its maximum value $N$, Brena-Gigli-Honda-Zhu obtain the following result (see Theorem 1.3 and Theorem 2.20 in \cite{brena2023weakly}).
\begin{thm}\label{thm2.21}
	If an RCD$(K,N)$ space $(X,d,m)$ satisfies $\dim_{ess}(X)=N$, then $m=c\mathcal{H}^{N}$ for some constant $c>0$.\ In particular, $(X,d,\mathcal{H}^{N})$ is an RCD$(K,N)$ space.
\end{thm}
Finally, let us recall that in dimension 2, the synthetic notions of lower bounds on sectional and Ricci curvature coincide (see \cite{lytchak2022ricci}).
\begin{thm}\label{thm2.22}
	If $(X,d,\mathcal{H}^{2})$ is an RCD$(K,2)$ space, then $(X,d)$ is an Alexandrov space with curvature $\geqslant K$.
\end{thm}
\noindent A combination of Theorem \ref{thm2.21}, Theorem \ref{thm2.22} and the results in \cite{kitabeppu2016characterization} enable us to handle the case $N=2$ of Theorem \ref{thm1.1}.

\section{Halfway Lemma and Uniform Cut Lemma on RCD(0,\textit{N}) Spaces}\label{3}
In this section, we extend two technical lemmas given by Sormani in \cite{sormani2000nonnegative} to a non-smooth context.\ First of all, we recall a notion introduced by Sormani \cite{sormani2000nonnegative}.
\begin{defn}\label{defn3.1}
	Let $X$ be a geodesic metric space which admits a universal cover $\tilde{X}$.\ Given $g\in G(X)$, we say that $\gamma$ is a minimal representative geodesic loop of $g$ if $\gamma=\pi\circ\tilde{\gamma}$, where $\tilde{\gamma}$ is a minimal geodesic from $\tilde{x}_{0}$ to $g\tilde{x}_{0}$.
\end{defn}
\begin{lem}[Halfway Lemma]\label{lem3.2}
	Let (X,d) be a proper geodesic metric space.\ Assume that (X,d) admits a universal cover ($\tilde{X}$,$\tilde{d}$).\ Then there exists an ordered set of independent generators $\{ g_{1},g_{2},g_{3},...\}$ of G(X) with minimal representative geodesic loops $\gamma_{k}$ of length $d_{k}$ such that 
	\[d(\gamma_{k}(0),\gamma_{k}(d_{k}/2))=d_{k}/2.\]
	If G(X) is infinitely generated, then we obtain a sequence of such generators.
\end{lem}
\begin{proof}
	Fix $x_{0}\in X$ and let $\tilde{x}_{0}\in\tilde{X}$ be a lift of $x_{0}$.\ Since $\tilde{X}$ is proper and $G:=G(X)$ acts discretely on $\tilde{X}$, there exists a non-trivial element $g_{1}\in G$ such that
	\[\tilde{d}(\tilde{x}_{0},g_{1}\tilde{x}_{0})=\min_{g\neq e}{\tilde{d}(\tilde{x}_{0},g\tilde{x}_{0})}>0.\]
	Let $G_{1}=\langle g_{1}\rangle$.\ Define each $g_{k}\in G$ and $G_{k}$ inductively by
	\[	\tilde{d}(\tilde{x}_{0},g_{k}\tilde{x}_{0})=\min_{g\in G\setminus G_{k-1}}{\tilde{d}(\tilde{x}_{0},g\tilde{x}_{0})}>0,\]
	\[G_{k}=\langle g_{1},...,g_{k}\rangle.\]
	Notice that $G\setminus G_{k-1}$ is nonempty for all $k$ if $G$ is infinitely generated.\ Let $\tilde{\gamma}_{k}:[0,d_{k}]\to\tilde{X}$ be a unit speed minimal geodesic from $\tilde{x}_{0}$ to $g_{k}\tilde{x}_{0}$.\ Define $\gamma_{k}(t):=\pi(\tilde{\gamma}_{k}(t))$ (i.e., $\gamma_{k}$ is the minimal representative geodesic loop of $g_{k}$ based at $x_{0}$).\\ 
	\indent It only remains to prove: $d(\gamma_{k}(0),\gamma_{k}(d_{k}/2))=d_{k}/2,\  \forall k.$\\
	\indent Suppose that there is a $k\in\mathbb{N}$ such that $d(\gamma_{k}(0),\gamma_{k}(d_{k}/2))<d_{k}/2$.\ Then there exists $T<d_{k}/2$ such that $d(\gamma_{k}(0),\gamma_{k}(T))<T$.\ Thus we can find a minimal geodesic $\sigma$ from $\gamma_{k}(T)$ to $\gamma_{k}(0)$ with length $<T$.\\
	\indent Denote $h_{1}=[\sigma\circ\gamma_{k}(0\to T)]\in G$ and $h_{2}=[\sigma\circ\gamma_{k}(d_{k}\to T)]\in G$.\ Then
	\[\tilde{d}(\tilde{x}_{0},h_{1}\tilde{x}_{0})\leqslant T+L(\sigma)<2T<d_{k},\] and
	\[\tilde{d}(\tilde{x}_{0},h_{2}\tilde{x}_{0})\leqslant d_{k}-T+L(\sigma)<d_{k}.\]
	Therefore, $h_{1},h_{2}\in G_{k-1}$.\ But $g_{k}=h_{2}^{-1}\circ h_{1}$, which is a contradiction since $g_{k}\in G\setminus G_{k-1}$. 
\end{proof}
\begin{lem}[Uniform Cut Lemma]\label{lem3.3}
	Let (X,d,m) be a RCD(0,N) space with $N>2$.\ Let $\gamma$ be a geodesic loop based at $x_{0}\in X$ with $L(\gamma)=d$ and $[\gamma]\in G(X)$ is nontrivial.\ Suppose $\gamma$ satisfies the following two conditions:
	\begin{enumerate}
		\item If $\sigma$ based at $ x_{0} $ is a loop such that $ [\sigma]=[\gamma]$ in $G(X)$, then $ L(\sigma)\geqslant d $
		\item $\gamma$ is minimal on $[0,d/2]$ and $[d/2,d]$.
	\end{enumerate}
    Then there is a universal constant $S_{N}$ defined in Theorem \ref{thm1.1}, such that for any $x\in\partial B_{rd}(x_{0})$ with $r\geqslant S_{N}+1/2$,
    \[d(x,\gamma(d/2))\geqslant(r-1/2)d+2S_{N}d.\]
\end{lem}
\begin{proof}
We first prove this lemma for $r=S_{N}+1/2$ and argue by contradiction.\ Suppose that there exists a point $x\in\partial B_{rd}(x_{0})$ such that 
\[L:=d(x,\gamma(d/2))<3S_{N}d.\]
Let $\alpha:[0,L]\to X$ be a unit speed minimal geodesic from $\gamma(d/2)$ to $x$.\ Let $(\tilde{X}, \tilde{d},\tilde{m},\tilde{x}_{0})$ be the universal cover of $(X,d,m,x_{0})$ and $\tilde{\gamma}$ be the lift of $\gamma$.\ By Theorem \ref{thm2.13}, $(\tilde{X}, \tilde{d},\tilde{m})$ is an RCD($0,N$) space.\\
\indent Denote $g=[\gamma]\in G(X)$.\ Then $\tilde{\gamma}$ is a minimal geodesic from $\tilde{x}_{0}$ to $g\tilde{x}_{0}$ by the first condition of $\gamma$.\ Thus $\tilde{d}(\tilde{x}_{0},g\tilde{x}_{0})=d$.\ We can also lift the curve $\alpha\circ\gamma(0\to d/2)$ to $\tilde{\alpha}\circ\tilde{\gamma}(0\to d/2)$ where $\tilde{\alpha}$ runs from $\tilde{\gamma}(d/2)$ to $\tilde{x}\in\tilde{X}$.\ Note that $L(\tilde{\alpha})=L$ and
\begin{align*}
	l_{1}:=\tilde{d}(\tilde{x},\tilde{x}_{0})\geqslant d(x,x_{0})=rd\\
	l_{2}:=\tilde{d}(\tilde{x},g\tilde{x}_{0})\geqslant d(x,x_{0})=rd
\end{align*}
The excess of $\tilde{x}$ w.r.t $\tilde{x}_{0}$ and $g\tilde{x}_{0}$ satisfies
\begin{align*}
	e(\tilde{x})=l_{1}+l_{2}-\tilde{d}(\tilde{x}_{0},g\tilde{x}_{0})\geqslant2rd-d=2S_{N}d.
\end{align*}
We can now apply the excess estimate (Theorem \ref{thm2.6}).\\
\indent Notice that
\begin{align}
	2S_{N}d\leqslant e(\tilde{x})\leqslant2\frac{N-1}{N-2}\left(\frac{N-1}{N}\frac{h^{N}}{l-h}\right)^{\frac{1}{N-1}},
\end{align}
where $l:=\min\{l_{1},l_{2}\},h=\min\tilde{d}(\tilde{x},\tilde{\gamma}(t))$ and
\begin{align}
	h & \leqslant L(\tilde{\alpha})=L<3S_{N}d,\\
	l-h & \geqslant(S_{N}+1/2)d-L>(1/2-2S_{N})d>d/4.
\end{align}
Combining (3.1)-(3.3), we get
\[S_{N}<\frac{N-1}{N-2}\left(4\frac{N-1}{N}(3S_{N})^N\right)^{\frac{1}{N-1}}\] and
\[S_{N}>\dfrac{N}{4(N-1)} \dfrac{1}{3^N} \left(\dfrac{N-2}{N-1}\right)^{N-1}.\]
This contradicts with the definition of $S_{N}$ ($N>2$).\\
\indent For $r>S_{N}+1/2$ and $x\in\partial B_{rd}(x_{0})$, let $y\in\partial B_{(1/2+S_{N})d}(x_{0})$ be a point on a minimal geodesic from $\gamma(d/2)$ to $x$.\ Then
\begin{align*}
	d(x,\gamma(d/2)) & =d(x,y)+d(y,\gamma(d/2))\\
	                 & \geqslant(rd-(1/2+S_{N})d)+3S_{N}d=(r-1/2)d+2S_{N}d.
\end{align*}
Thus we complete the proof.
\end{proof}

\section{Proof of Main Theorems}\label{4}
In this section, we prove Theorem \ref{thm1.1}, Corallary \ref{cor1.2} and Corallary \ref{cor1.3}.
	\begin{proof}[Proof of Theorem \ref{thm1.1}]
	We first point out that there is a complete classification for RCD$^*$($K,N$) spaces when $N\in[1,2)$.\ Indeed, $X$ is isometric to $\mathbb{R},\mathbb{R}_{+},S^{1}(r)$ or $[0,l]$ (see Corollary 1.2 in \cite{kitabeppu2016characterization}).\ Thus we only need to consider the case $N\geqslant2$.\\
	\indent (1) $N>2$.\\
	\indent Suppose $G:=G(X)\cong\pi_{1}(X)$ is infinitely generated.\ Construct a sequence of independent generators $\{g_{k}\}$, as in Lemma \ref{lem3.2}, with minimal representative geodesic loops $\gamma_{k}$ based at some point $p$.\ Notice that $\gamma_{k}$ satisfies the hypothesis in Lemma \ref{lem3.3}.\ Let $(\tilde{X},\tilde{d},\tilde{p})$ be the universal cover of $(X,d,p)$.\\
	\indent We observe that $d_{k}:=L(\gamma_{k})$ diverges to infinity, since otherwise the orbit $G\tilde{p}$ would be contained in a closed ball $\bar{B}_{R}(\tilde{p})\subset\tilde{X}$.\ Since $\bar{B}_{R}(\tilde{p})$ is compact and $G$ acts discretely on $\tilde{X}$, $G$ must be finite which is a contradiction.\\
	\indent Choose a sequence $x_{k}\in\partial B_{(\frac{1}{2}+S_{N})d_{k}}(p)$, then by Lemma \ref{lem3.3},
	\[d(x_{k},\gamma_{k}(d_{k}/2))\geqslant3S_{N}d_{k}.\]
	There exists a point $y_{k}\in\partial B_{\frac{1}{2}d_{k}}(p)$ on the minimal geodesic from $p$ to $x_{k}$ and $y_{k}$ satisfies
	\begin{align*}
		d(y_{k},\gamma_{k}(d_{k}/2)) & \geqslant d(x_{k},\gamma_{k}(d_{k}/2))-d(x_{k},y_{k})\\
		                             & \geqslant 3S_{N}d_{k}-S_{N}d_{k}
		                             =2S_{N}d_{k}.
	\end{align*}
    Then,
    \begin{align*}
    	\limsup\limits_{r\to\infty}\frac{diam(\partial B_{r}(p))}{r} & \geqslant 	\limsup\limits_{k\to\infty}\frac{d(y_{k},\gamma_{k}(d_{k}/2))}{d_{k}/2}\\
    	& \geqslant \limsup\limits_{k\to\infty}\frac{2S_{N}d_{k}}{d_{k}/2}=4S_{N},
    \end{align*}
    which is a contradiction.\\
    \indent (2) $N=2$.\\
    \indent In this case, $\dim_{ess}(X)=$ 1 or 2.\ If $\dim_{ess}(X)=1$, then $\mathcal{R}_{1}\neq\emptyset$ and $X$ is isometric to $\mathbb{R},\mathbb{R}_{+},S^{1}(r)$ or $[0,l]$ (see Theorem 1.1 in \cite{kitabeppu2016characterization}).\ Thus, we may assume $\dim_{ess}(X)=2$.\ By Theorem \ref{thm2.21} and Theorem \ref{thm2.22}, we know that $(X,d)$ is a 2-dimensional Alexandrov space with curvature $\geqslant0$.\ Then we may go through Gromov's arguments in \cite{gromov1978manifolds}.\\
    \indent Let $(\tilde{X},\tilde{d},\tilde{x})$ (which is also an Alexandrov space with curvature $\geqslant0$) be the universal cover of $(X,d,x)$.\ Notice that $G:=\pi_{1}(X,x)$ acts on $\tilde{X}$ via isometries.\ Without loss of generality, we assume that $\tilde{x}\in\tilde{X}$ is a regular point (i.e., at $\tilde{x}$, the space of directions $\Sigma_{\tilde{x}}$ is isometric to standard $S^{1}$, or equivalently, the tangent cone $C(\Sigma_{\tilde{x}})$ is isometric to $\mathbb{R}^{2}$).\\
    \indent Choose a generating set $\{g_{1},g_{2},...\}$ inductively such that
    \begin{enumerate}
    	\item $\tilde{d}(\tilde{x},g_{1}\tilde{x})\leqslant\tilde{d}(\tilde{x},g\tilde{x})$ for all $g\in G\setminus\{e\}$,
    	\item $g_{k}\in G\setminus\langle g_{1},...,g_{k-1}\rangle$, $k\geqslant2$,
    	\item $\tilde{d}(\tilde{x},g_{k}\tilde{x})\leqslant\tilde{d}(\tilde{x},g\tilde{x})$ for all $g\in G\setminus\{g_{1},...,g_{k-1}\}, k\geqslant2$.
    \end{enumerate}
    Clearly, $\tilde{d}(\tilde{x},g_{k}\tilde{x})\leqslant\tilde{d}(\tilde{x},g_{l}\tilde{x})$ if $k<l$.\ Let $\gamma_{k}$ be the minimal geodesic from $\tilde{x}$ to $g_{k}\tilde{x}$.\ We claim that $\angle(\gamma_{k},\gamma_{l})\geqslant\frac{\pi}{3}$ for $k<l$.\ Otherwise, consider the comparison triangle in $\mathbb{R}^{2}$.\ We get that $\tilde{d}(g_{l}\tilde{x},g_{k}\tilde{x})<\tilde{d}(\tilde{x},g_{l}\tilde{x})$.\ But then,
    \[\tilde{d}(\tilde{x},g_{l}^{-1}g_{k}\tilde{x})<\tilde{d}(\tilde{x},g_{l}\tilde{x}).\] This contradicts our choice of $g_{l}$.\\
    \indent Recall that an equivalent class of a minimal geodesic starting at $\tilde{x}$ is a direction at $\tilde{x}$ and the distance between two classes is the angle between them.\ Since $\Sigma_{\tilde{x}}$ is isometric to $S^{1}$ and $\angle(\gamma_{k},\gamma_{l})\geqslant\frac{\pi}{3}$ for $k\neq l$, the generating set we construct above contains at most 6 elements, i.e.
    \[G=\langle g_{1},...,g_{s}\rangle,\ s\leqslant6.\] Thus we complete the proof.
	\end{proof}
	
	\begin{proof}[Proof of corollary \ref{cor1.2}]
		Let $m_{i}=\text{Vol}(B_{1}(p_{i}))^{-1}\text{vol}_{g_{i}}$ be the normalized measure on $(M_{i}^{n},g_{i},p_{i})$.\ Then by Cheeger-Colding's seminal work \cite{cheeger1997structure, cheeger2000structure}, there exists a Radon measure $m$ on $X$, such that $(M_{i}^{n},g_{i},m_{i},p_{i})\xrightarrow{pmGH}(X,d,m,p)$.\ By the Stability Theorem (Proposition \ref{prop2.5}), $(X,d,m)$ is an RCD($0,n$) space.\ Thus, we get the conclusion directly by Theorem \ref{thm1.1}.
	\end{proof}

\begin{proof}[Proof of Corollary \ref{cor1.3}]
	As in Theorem \ref{thm1.1}, we only need to consider the case $N>2$.\ If $(X,d,m)$ does not split, then by Theorem \ref{thm2.8} and Theorem \ref{thm1.1}, we get the conclusion.\\
	\indent If $(X,d,m)$ splits, then by Theorem \ref{thm2.7}, $(X,d,m)$ is isomorphic to $(X'\times\mathbb{R},d'\times d_{E},m'\times\mathcal{L}^{1})$, where $(X',d',m')$ is an RCD($0,N-1$) space.\ We claim that $X'$ is compact.\ Otherwise $m'(B_{r}(x'))\geqslant Cr$ by Theorem \ref{thm2.9} and hence, $m(B_{r}(x))\geqslant Cr^{2}$ where $x=(x',0)\in X'\times\mathbb{R}$.\ This is a contradiction with our linear volume growth assumption.\\
	\indent Now suppose $G(X)=G(X'\times\mathbb{R})\cong\pi_{1}(X)$ is infinitely generated.\ We apply Lemma \ref{lem3.2} to obtain a sequence of independent generators $\{ g_{1},g_{2},g_{3},...\}$ of $G(X)$ with minimal representative geodesic loops $\gamma_{k}$ of length $d_{k}\to\infty$, and
	\[d(\gamma_{k}(0),\gamma_{k}(d_{k}/2))=d_{k}/2.\]
	Note that $\gamma_{k}$ is of the form $\sigma_{k}^{1}\times\sigma_{k}^{2}$, where $\sigma_{k}^{1}$ is a loop in $X'$ and $\sigma_{k}^{2}$ is a loop in $\mathbb{R}$.\ Thus $\gamma_{k}$ is homotopic to $\sigma_{k}^{1}\times0$ and $[\gamma_{k}]=[\sigma_{k}^{1}\times0]$ in $G(X)$.\ Since $L(\gamma_{k})$ is minimal in the equivalent class, $\gamma_{k}$ must be of the form $\sigma_{k}^{1}\times0$.\\
	\indent Therefore, we get a sequence of loops $\{\sigma_{k}^{1}\}$ on $(X',d')$ with
	\[d'(\sigma_{k}^{1}(0),\sigma_{k}^{1}(d_{k}/2))=d_{k}/2.\] 
	But $X'$ is compact, which is a contradiction to $d_{k}\to\infty$.
\end{proof}

$\mathbf{\boldsymbol{\mathbf{Acknowledgement.}\mathbf{}}}$
The author would like to thank Bobo Hua for helpful suggestions on this research project.\ The author would also like to thank Andrea Mondino for useful comments on an earlier version of this paper. 
	\nocite{*}
	\bibliographystyle{siam}
	\bibliography{ref} 

\end{document}